\title{Compact graphings}
\author{L\'aszl\'o Lov\'asz\\
Hungarian Academy of Sciences and E\"otv\"os Lor\'and University\\
Budapest}
\date{\today}

\documentclass[11pt,fleqn]{article}
\usepackage{amsmath,amssymb,bbm,bm,delarray,ntheorem}
\usepackage[hypertex]{hyperref}
\long\def\ignore#1{}

\begin{document}

\newtheorem{theorem}{Theorem}
\newtheorem{prop}[theorem]{Proposition}
\newtheorem{lemma}[theorem]{Lemma}
\newtheorem{claim}{Claim}
\newtheorem{corollary}[theorem]{Corollary}
\theorembodyfont{\rmfamily}
\newtheorem{remark}[theorem]{Remark}
\newtheorem{example}{Example}
\newtheorem{conj}{Conjecture}
\newtheorem{problem}[theorem]{Problem}
\newtheorem{step}{Step}
\newtheorem{alg}{Algorithm}
\newenvironment{proof}{\medskip\noindent{\bf Proof. }}{\hfill$\square$\medskip}

\def\R{\mathbb{R}}
\def\one{\mathbbm1}
\def\T{^{\sf T}}
\def\Pr{{\sf P}}
\def\E{{\sf E}}
\def\Q{{\mathbf Q}}
\def\bd{\text{bd}}
\def\eps{\varepsilon}
\def\wh{\widehat}
\def\cork{\text{\rm corank}}
\def\rank{\text{\rm rank}}
\def\Ker{\text{\rm Ker}}
\def\rk{\text{\rm rank}}
\def\supp{\text{\rm supp}}
\def\diag{\text{\rm diag}}
\def\sep{\text{\rm sep}}
\def\tr{\text{\rm tr}}
\def\iso{{h}}

\def\AA{\mathcal{A}}\def\BB{\mathcal{B}}\def\CC{\mathcal{C}}
\def\DD{\mathcal{D}}\def\EE{\mathcal{E}}\def\FF{\mathcal{F}}
\def\GG{\mathcal{G}}\def\HH{\mathcal{H}}\def\II{\mathcal{I}}
\def\JJ{\mathcal{J}}\def\KK{\mathcal{K}}\def\LL{\mathcal{L}}
\def\MM{\mathcal{M}}\def\NN{\mathcal{N}}\def\OO{\mathcal{O}}
\def\PP{\mathcal{P}}\def\QQ{\mathcal{Q}}\def\RR{\mathcal{R}}
\def\SS{\mathcal{S}}\def\TT{\mathcal{T}}\def\UU{\mathcal{U}}
\def\VV{\mathcal{V}}\def\WW{\mathcal{W}}\def\XX{\mathcal{X}}
\def\YY{\mathcal{Y}}\def\ZZ{\mathcal{Z}}

\def\Ab{\mathbf{A}}\def\Bb{\mathbf{B}}\def\Cb{\mathbf{C}}
\def\Db{\mathbf{D}}\def\Eb{\mathbf{E}}\def\Fb{\mathbf{F}}
\def\Gb{\mathbf{G}}\def\Hb{\mathbf{H}}\def\Ib{\mathbf{I}}
\def\Jb{\mathbf{J}}\def\Kb{\mathbf{K}}\def\Lb{\mathbf{L}}
\def\Mb{\mathbf{M}}\def\Nb{\mathbf{N}}\def\Ob{\mathbf{O}}
\def\Pb{\mathbf{P}}\def\Qb{\mathbf{Q}}\def\Rb{\mathbf{R}}
\def\Sb{\mathbf{S}}\def\Tb{\mathbf{T}}\def\Ub{\mathbf{U}}
\def\Vb{\mathbf{V}}\def\Wb{\mathbf{W}}\def\Xb{\mathbf{X}}
\def\Yb{\mathbf{Y}}\def\Zb{\mathbf{Z}}

\def\ab{\mathbf{a}}\def\bb{\mathbf{b}}\def\cb{\mathbf{c}}
\def\db{\mathbf{d}}\def\eb{\mathbf{e}}\def\fb{\mathbf{f}}
\def\gb{\mathbf{g}}\def\hb{\mathbf{h}}\def\ib{\mathbf{i}}
\def\jb{\mathbf{j}}\def\kb{\mathbf{k}}\def\lb{\mathbf{l}}
\def\mb{\mathbf{m}}\def\nb{\mathbf{n}}\def\ob{\mathbf{o}}
\def\pb{\mathbf{p}}\def\qb{\mathbf{q}}\def\rb{\mathbf{r}}
\def\sb{\mathbf{s}}\def\tb{\mathbf{t}}\def\ub{\mathbf{u}}
\def\vb{\mathbf{v}}\def\wb{\mathbf{w}}\def\xb{\mathbf{x}}
\def\yb{\mathbf{y}}\def\zb{\mathbf{z}}

\def\Abb{\mathbb{A}}\def\Bbb{\mathbb{B}}\def\Cbb{\mathbb{C}}
\def\Dbb{\mathbb{D}}\def\Ebb{\mathbb{E}}\def\Fbb{\mathbb{F}}
\def\Gbb{\mathbb{G}}\def\Hbb{\mathbb{H}}\def\Ibb{\mathbb{I}}
\def\Jbb{\mathbb{J}}\def\Kbb{\mathbb{K}}\def\Lbb{\mathbb{L}}
\def\Mbb{\mathbb{M}}\def\Nbb{\mathbb{N}}\def\Obb{\mathbb{O}}
\def\Pbb{\mathbb{P}}\def\Qbb{\mathbb{Q}}\def\Rbb{\mathbb{R}}
\def\Sbb{\mathbb{S}}\def\Tbb{\mathbb{T}}\def\Ubb{\mathbb{U}}
\def\Vbb{\mathbb{V}}\def\Wbb{\mathbb{W}}\def\Xbb{\mathbb{X}}
\def\Ybb{\mathbb{Y}}\def\Zbb{\mathbb{Z}}

\def\Af{\mathfrak{A}}\def\Bf{\mathfrak{B}}\def\Cf{\mathfrak{C}}
\def\Df{\mathfrak{D}}\def\Ef{\mathfrak{E}}\def\Ff{\mathfrak{F}}
\def\Gf{\mathfrak{G}}\def\Hf{\mathfrak{H}}\def\If{\mathfrak{I}}
\def\Jf{\mathfrak{J}}\def\Kf{\mathfrak{K}}\def\Lf{\mathfrak{L}}
\def\Mf{\mathfrak{M}}\def\Nf{\mathfrak{N}}\def\Of{\mathfrak{O}}
\def\Pf{\mathfrak{P}}\def\Qf{\mathfrak{Q}}\def\Rf{\mathfrak{R}}
\def\Sf{\mathfrak{S}}\def\Tf{\mathfrak{T}}\def\Uf{\mathfrak{U}}
\def\Vf{\mathfrak{V}}\def\Wf{\mathfrak{W}}\def\Xf{\mathfrak{X}}
\def\Yf{\mathfrak{Y}}\def\Zf{\mathfrak{Z}}

\maketitle

%\tableofcontents

\begin{abstract}
Graphings are special bounded-degree graphs on probability spaces, representing
limits of graph sequences that are convergent in a local or local-global sense.
We describe a procedure for turning the underlying space into a compact metric
space, where the edge set is closed and nearby points have nearby graph
neighborhoods.
\end{abstract}

\section{Introduction}\label{SEC:PRELIM}

Graphings are bounded-degree graphs on probability spaces, representing limits
of graph sequences that are convergent in a local or local-global sense. For
many aspects of this theory, we refer to \cite{HomBook}.

For dense graphs, a theory that is analogous is many respects (but different in
the details) is the theory of graphons. One of the important tools in graphon
theory is the procedure of ``purification'' of a graphon, which results in a
graphon that is essentially equivalent in all important properties and
parameters, and in addition it is defined on a compact metric space
\cite{LSz3,LSz8}.

The aim of this note is to prove a similar result for graphings: every graphing
can be turned into an equivalent graphing (differing only on a null set) that
is defined on a compact metric space, whose edge set is closed, and if two
points are close in this metric, then large neighborhoods of them (in the graph
distance) are isomorphic and point-by-point close.

This construction has applications, among others, in the theory of
hyperfiniteness; this is published in another paper. Other constructions
compactifying graphings have been given (in a group theory setting) by Elek
\cite{Elek1,Elek2,Elek3}

\section{Graphings}

We fix a positive integer $D$, and all graphs we consider are tacitly assumed
to have maximum degree at most $D$.

A {\it graphing} $\Gb=(I,\AA,E,\lambda)$ is a simple graph (with all degrees
bounded by $D$) on node set $I$, where $(I,\AA)$ is a standard Borel space,
$\lambda$ is a probability measure on $(I,\AA)$, the set of edges $E\subseteq
I\times I$ is a Borel set avoiding the diagonal of $I\times I$ and invariant
under interchanging the coordinates, and the following ``measure-preservation''
condition is satisfied for any two subsets $A,B\in\AA$:
\begin{equation}\label{EQ:UNIMOD}
\int\limits_A \deg_B(x)\,d\lambda(x) =\int\limits_B \deg_A(x)\,d\lambda(x).
\end{equation}
Here $\deg_B(x)$ denotes the number of edges connecting $x\in I$ to points of
$B$. (It can be shown that this is a Borel function of $x$.) Most of the time,
we may assume that $I=[0,1]$ and $\lambda$ is the Lebesgue measure. In other
cases, $I$ will be a complete separable metric space, and $\AA$, the set of
Borel subsets of $I$. We suppress the sigma-algebra $\AA$ in our notation most
of the time.

Such a graphing defines a measure on Borel subsets of $I^2$: on rectangles we
define
\[
\eta(A\times B) = \int\limits_A \deg_B(x)\,d\lambda(x),
\]
which extends to Borel subsets in the standard way. We call this the {\it edge
measure} of the graphing. It is concentrated on the set of edges, and it is
symmetric in the sense that interchanging the two coordinates does not change
it.

Most of the time, we will equip $I$ with a metric $d$ so that $\AA$ becomes the
set of Borel subsets of $(I,d)$. If this is the case, we write
$\Gb=(I,d,E,\lambda)$. Open balls in this metric space will be denoted by
$\Bf(x,\eps)=\Bf_\Gb(x,\eps)$, with radius $0<\eps\le 1$ and center $x\in I$.

We have another metric on $\Gb$, the distance in the graph. We denote the
distance of two points $x,y\in I$ in the graph by $d_\Gb(x,y)$, and define
$d_\Gb(x,y)=\infty$ if $x$ and $y$ belong to different connected components.
Let us define  an {\it $r$-ball} as a rooted graph with degrees at most $D$ and
radius (maximum graph distance from the root) at most $r$. In a graphing $\Gb$,
let $B(x,r)=B_\Gb(x,r)$ denote the subgraph induced by nodes at graph distance
at most $r$ from $x$.

We say that a graphing $\Gb=(I,\AA,E,\lambda)$ is a {\it full subgraphing} of a
graphing $\Gb'=(I',\AA',E',\lambda')$, if $\Gb$ is the union of connected
components of $\Gb'$, $I$ is a Borel subset of $I'$, $\AA=\AA'|_I$, and
$\lambda'(X)=\lambda(I\cap X)$ for every Borel subset $X\subseteq I'$. In
particular, $\lambda'(I)=\lambda(I)=1$, so $\lambda'(I'\setminus I)=0$.

For a graphing $\Gb$, let us pick a random element $x\in I$ according to
$\lambda$, and consider the connected component $\Gb_x$ of $\Gb$ containing
$x$. Since all degrees are finite (bonded by $D$), the graph $\Gb_x$ is finite
or countably infinite. This gives us a probability distribution $\rho_{\Gb_1}$
on rooted connected graphs with degrees bounded by $D$. (The underlying
sigma-algebra is generated by the sets obtained by fixing a finite neighborhood
of the root.) Such distributions, occurring under many names (unimodular,
involution invariant) were introduced in the path-breaking work of Benjamini
and Schramm \cite{BS2} defining convergence of bounded-degree graph sequences.

Two graphings $\Gb_1$ and $\Gb_2$ are called {\it locally equivalent}, if
$\rho_{\Gb_1}=\rho_{\Gb_2}$. A stronger notion of equivalence of graphings,
called {\it local-global equivalence}, was defined in \cite{HLSz} (we don't
define this somewhat more complicated notion here). A full subgraphing is
equivalent to the original graphing both in the local and local-global sense.

\section{Compactifying graphings}

Let $\Gb=(I,E,\lambda)$ be a graphing. For two points $x,y\in I$ and integer
$r\ge 0$, we define an {\it $r$-neighborhood isomorphism} between $x$ and $y$
as an isomorphism $\phi:~B(x,r)\to B(y,r)$ such that $\phi(x)=y$. Such an
isomorphism always exists for $r=0$. For an $r$-neighborhood isomorphism $\phi$
between $x$ and $y$, we define
\begin{equation}\label{EQ:DIST}
d(\phi) = \max_{z\in B(x,r)} d(z,\phi(z)).
\end{equation}

We say that a graphing $\Gb=(I,d,E,\lambda)$ is {\it compact}, if the following
conditions hold:

\smallskip

(C1) $(I,d)$ is a compact metric space;

\smallskip

(C2) $E$ is a closed subset of $I\times I$;

\smallskip

(C3) for every real $\eps>0$ and integer $r\ge 0$ there is a $\delta>0$ such
that for every pair $x,y\in I$ with $d(x,y)\le\delta$ there exists an
$r$-neighborhood isomorphism $\phi$ between $x$ and $y$ with $d(\phi)\le \eps$.

The main result in this note is the following theorem.

\begin{theorem}\label{THM:COMP}
Every graphing is a full subgraphing of a compact graphing.
\end{theorem}

The compact graphing constructed below will be called a {\it compactification}
of the original graphing. It is easy to see (from the proof or by direct
manipulation) that one could make further assumptions like that the diameter of
$(I,d)$ is $1$ and $d(x,y)=1$ for every edge $xy$. Also, the dependence of
$\delta$ on $\eps$ and $r$ in (C3) is mild: $\delta=\eps/(1+r\eps)$ will
suffice.

\begin{proof}
Let us start with an appropriate compact metric on $I$:

\begin{claim}\label{CLAIM:COMP-EDGE}
There is a metric $d_0$ on $I$ such that $d_0$ defines the given Borel sets,
$(I,d_0)$ is compact, $d_0$ has diameter $1$, and if $xy\in E$, then
$d_0(x,y)=1$.
\end{claim}

By a theorem of Kechris, Solecki and Todorcevic \cite{KST}, the graphing has a
Borel coloring $\{B_0,\dots,B_D\}$ with $D+1$ colors. On every color class,
define a compact metric with diameter $1$ giving the prescribed Borel sets (for
example, using an isomorphism between $(B_0,\AA|_{B_0})$ and the Borel
sigma-algebra of the interval $[0,1]$ with the euclidean metric). Define the
distance of two points in different color classes to be $1$. This metric $d_0$
satisfies the requirements of the claim.

\smallskip

Next, we define a new metric on $I$ by
\begin{equation}\label{EQ:DIST2}
d(x,y)=\inf\max\Bigl\{\frac1{r+1},d_0(\phi)\Bigr\},
\end{equation}
where the infimum is taken over all $r\ge 0$ and over all $r$-neighborhood
isomorphisms $\phi$ between $x$ and $y$. Note that the infimum is attained for
$r_0=\lceil 1/d(x,y)\rceil-1$: for smaller values of $r$, $1/(r+1)>d(x,y)$, and
for larger values of $r$, we can restrict the optimal $r$-neighborhood
isomorphism to $B(x,r_0)$.

It is easy to see that $d$ is a metric, $d_0(x,y)\le d(x,y)\le 1$, and
$d(x,y)=1$ for every edge $xy\in E$. Besides the balls $B(x,r)$ in the graph
metric, we have to distinguish the balls $\Bf_0(x,\eps)$ in the space $(I,d_0)$
and the balls $\Bf(x,\eps)$ in the space $(J,d)$.

Let $(J,d)$ be the completion of the metric space $(I,d)$.

\begin{claim}\label{CLAIM:COMP1}
$(J,d)$ is a compact metric space.
\end{claim}

It suffices to prove that every infinite sequence of points $(x_1,x_2,\dots)$
in $I$ has a subsequence that is Cauchy in the metric $d$. Since there are only
a bounded number of possible neighborhoods of any radius, we can select a
subsequence such that $B(x_n,r)$ is isomorphic to $B(x_r,r)$ for $0\le r\le n$.
We fix neighborhood isomorphisms $\phi_r:~B(x_r,r)\to B(x_{r+1},r)$, and let
$\phi_{r,n}=\phi_{n-1}\circ\dots\circ\phi_r$. Selecting a subsequence again, we
may assume that for every $r\ge 0$ and every $z\in B(x_r,r)$, the sequence
$(\phi_{r,n}(z):~n=r,r+1,\dots)$ is convergent in the $d_0$ metric.

Let $\eps>0$ and $r=\lceil 1/\eps\rceil-1$. Since $B(x_r,r)$ is finite, there
is an $N=N(\eps)$ such that for $n>m>N$ we have
$d_0(\phi_{r,n}(z),\phi_{r,m}(z))<\eps$ for every $z\in B(x_r,r)$. Let $\phi$
denote the restriction of $\phi_{n-1}\circ\dots\circ\phi_m$ to
$\phi_{r,m}(B(x_r,r))$. Then $\phi$ is an $r$-neighborhood isomorphism between
$x_m$ and $x_n$, and
\begin{align*}
d_0(\phi)&=\max_{z\in B(x_m,r)} d_0(z,\phi(z)) \le \eps.
\end{align*}
This shows that $d(x_n,x_m)<\eps$. So $(x_1,x_2,\dots)$ is a Cauchy sequence,
which proves the claim.

\begin{claim}\label{CLAIM:COMP3}
The Borel sets in $(I,d_0)$ are exactly the restrictions of Borel sets in
$(J,d)$ to $I$.
\end{claim}

The fact that $d_0\le d$ implies that if a set is open in $(I,d_0)$, then it is
open in $(I,d)$, and so it is the restriction of an open set of $(J,d)$ to $I$.
Restrictions of Borel sets of $(J,d)$ to $I$ form a sigma-algebra, which
contains all open sets of $(I,d_0)$, and hence, all Borel sets.

The converse is a bit more elaborate. It suffices to show that for every open
set $U$ of $(J,d)$, the intersection $I\cap U$ is a Borel set in $(I,d_0)$ (not
necessarily open!). We may assume $U=\Bf(x,\eps)$ for some $x\in I$ and
$\eps>0$, since $U\cap I$ can be covered by a countable number of such balls
contained in $U$ (using that $(J,d)$ is a compact metric space and hence
separable).

Let $r=\lceil1/\eps\rceil-1$, and $B(x,r)=\{x=x_1,x_2,\dots,x_N\}$. Consider
the set $S$ of those points $(y_1,\dots,y_N)\in I^N$ for which
$\phi:~x_i\mapsto y_i$ is an $r$-neighborhood isomorphism between $x_1$ and
$y_1$ with $d_0(\phi)\le\eps$. It is clear that $S$ is a Borel set in
$(I,d_0)^N$. If $(y_1,\dots,y_N)\in S$, then the points $y_1,\dots,y_N$ induce
a connected subgraph of $\Gb$, and so every point $y_1$ is contained in a
finite number of such $N$-tuples. Thus Lusin's Theorem implies that the
projection of $S$ to the first coordinate, which is just the set $I\cap U$, is
Borel.

\begin{claim}\label{CLAIM:COMP4}
The set $I$ is Borel in $(J,d)$.
\end{claim}

Indeed, an immediate consequence of Claim \ref{CLAIM:COMP3} is that the
embedding $f:~(I,d_0) \hookrightarrow (J,d)$ is Borel measurable. Claim
\ref{CLAIM:COMP1} implies that $(J,d)$ defines a standard Borel space, and by
hypothesis, so does $(I,d_0)$. It follows that the graph of $f$ is Borel in
$(I,d_0)\times(J,d)$, and so the set $I$, which is its projection to $J$, is a
Borel set in $(J,d)$ by Lusin's Theorem.

We define a graphing $\wh\Gb$ on $J$. The underlying probability measure
$\wh{\lambda}$ is defined on the Borel subsets of $(J,d)$ by
\[
\wh{\lambda}(X) = \lambda(X\cap I).
\]
We define the edge set $\wh{E}$ as the closure of $E$ in $(J,d)^2$. This
guarantees that $\wh{E}$ is a symmetric Borel set in $(J,d)^2$. Furthermore,
$\wh{E}$ does not meet the diagonal of $J\times J$; in fact every edge has
length $1$ in the metric $d$.

Next we show that $\Gb$ is the union of connected components of $\wh{\Gb}$
contained in $I$.

\begin{claim}\label{CLAIM:COMP2}
If $xy\in\wh{E}$ and $x\in I$, then $y\in I$ and $xy\in E$.
\end{claim}

By the definition of $\wh{E}$, there are edges $x_ny_n\in E$ such that $x_n\to
x$ and $y_n\to y$ in $(J,d)$. There are $r_n$-neighborhood isomorphisms
$\phi_n$ in $\Gb$ from $x_n$ to $x$ such that $r_n\to\infty$ and
$d_0(\phi_n)\to 0$. This implies that $y_n'=\phi_n(y_n)$ exists for
sufficiently large $n$ and $xy_n'\in E$. Since $x$ has a finite number of
neighbors, we can select an infinite subsequence for which $y_n'=y'$ is
independent of $n$. Then
\[
d(y_n,y')=d(y_n,\phi_n(y_n))\le \max\Bigl\{\frac1{r_n+1},d_0(\phi_n)\Bigr\}\to 0.
\]
But $d(y_n,y)\to 0$ by hypothesis, and hence $y=y'\in I$ and $xy=xy'\in E$.
This proves the claim.

\begin{claim}\label{CLAIM:COMPD}
All degrees in $\wh{\Gb}$ are bounded by $D$.
\end{claim}

For the degree of $x\in I$ this follows by Claim \ref{CLAIM:COMP2}. Let $x\in
J\setminus I$, and let $xy_1,\dots,xy_r\in \wh{E}$. For every $i$, there are
sequences $(x_{i,n}:~n=1,2,\dots)$ and $(y_{i,n}:~n=1,2,\dots)$ of points in
$(I,d)$ such that $x_{i,n}\to x$, $y_{i,n}\to y_i$, and $x_{i,n}y_{i,n}\in E$.
We may assume that $d(x_{i,n},x)\le 1/2^n$ and $d(y_{i,n},y_i)\le 1/2^n$.

We may also assume that $x_{1,n}=\dots=x_{r,n}$ for all $r$. Indeed, we have
$d(x_{1,n},x_{i,n})\le d(x_{1,n},x)+d(x_{i,n},x)\le 1/2^{n-1}$, and hence (for
$n\ge 2$) we have a $1$-neighborhood isomorphism $\phi$ between $x_{i,n}$ and
$x_{1,n}$ such that $d(\phi)\le 1/2^{n-1}$. The point $y'_{i,n}=\phi(x_{i,n})$
is a neighbor of $x_{1,n}$. Furthermore, $d(y_{i,n},y'_{i,n}))\le 1/2^{n-1}$,
and hence $d(y_i,y'_{i,n}))\le 1/2^{n-2}$. So $y'_{i,n}\to y_i$, and we can
replace the edge $x_{i,n}y_{i,n}$ by $x_{1,n}y'_{i,n}$.

Now if the points $y_1,\dots,y_r$ are different, then the points
$y'_{1,n},\dots,y'_{r,n}$ are different for sufficiently large $n$. But these
points are neighbors of $x_{1,n}$ in $\Gb$, and hence $r\le D$, proving the
Claim.

It is easy to check, using that $\wh{\lambda}(J\setminus I)=0$, that
$\wh{\lambda}$ satisfies the identity \eqref{EQ:UNIMOD}, and so $\wh\Gb$ is a
graphing. We show that it is a compact graphing. We have seen that the
underlying space $J$ is a compact metric space and the edge set $E$ is a closed
subset of $J\times J$ by its definition. To conclude, it suffices to prove that
the third condition in the definition of compactness holds.

\begin{claim}\label{CLAIM:COMP7}
For every real number $\eps>0$, integer $r\ge 0$ and $x,y\in J$ with
$d(x,y)\le\eps/(1+r\eps)$, there exists an $r$-neighborhood isomorphism $\phi$
between $x$ and $y$ with $d(\phi)\le \eps$.
\end{claim}

Let $R=\lceil1/d(x,y)\rceil-1$. By the definition of $d$, there is an
$R$-neighborhood isomorphism $\psi$ between $x$ and $y$ such that $d_0(\psi)\le
d(x,y)<\eps$. For every $z\in B(x,r)$, we have $B(z,R-r)\subseteq B(x,R)$, and
so restricting $\psi$ to $B(z,R-r)$ we get an $(R-r)$-neighborhood isomorphism
$\psi_z$ between $z$ and $\psi(z)$. Hence
\[
d(z,\psi(z))\le \max\Big\{\frac1{R-r+1}, d_0(\psi_z)\Big\} \le
\max\Bigl\{\frac1{R-r+1},\,d_0(\psi)\Big\} \le \eps.
\]
So restricting $\psi$ to $B(x,r)$, we get an $r$-neighborhood isomorphism
$\phi$ with $d(\phi)\le\eps$. This proves the claim.
\end{proof}

We may do another ``purifying'' operation of $\wh\Gb$.

\begin{lemma}\label{LEM:}
Let  $\Gb=(I,\AA,\lambda,E)$  be a compact graphing, and let $S$ be the support
of the measure $\lambda$. Then the restriction $\Gb[S]$ of $\Gb$ to $S$ is a
compact graphing, which is a full subgraphing of $\Gb$.
\end{lemma}

\begin{proof}
The only nontrivial assertion is that no edge of $\Gb$ connects $S$ to
$I\setminus S$. Suppose that there is an edge $xy\in E$ for which $x\in S$ and
$y\in I\setminus S$. Since $S$ is closed, there is an $\eps>0$ such that
$\Bf(y,\eps)\subseteq I\setminus S$. By the definition of a compact graphing,
there is a $\delta>0$ such that if $d(x,z)<\delta$, then there is a
$1$-neighborhood isomorphism $\phi$ between $x$ and $z$ such that
$d(v,\phi(v))<\eps$ for every $v\in N(x)$. In particular, $d(y,\phi(y))<\eps$,
and so $\phi(y)\in I\setminus S$. So $z=\phi(v)$ has a neighbor in $S$. But
then
\[
\int_S \deg_{I\setminus S}(z)\,d\lambda(z) \ge \int_{\Bf(x,\delta)} 1\,d\lambda(z)>0
\]
(since $\Bf(x,\delta)$ is an open set not contained in $I\setminus S$, and so
it has positive measure), but
\[
\int_{I\setminus S} \deg_S(z)\,d\lambda(z)=0
\]
(since $\lambda(I\setminus S)=0$), which contradicts the measure preserving
property \eqref{EQ:UNIMOD} of graphings.
\end{proof}

The following example illustrates how these constructions work.

\begin{example}\label{EXA:CIRC}
Consider the graphing $C_\alpha$ defined on the unit circle by connecting two
points if their angular distance is $2\alpha\pi$. Then every connected
component of $C_\alpha$ is a 2-way infinite path. It is clear that $C_\alpha$
is a compact graphing, with the metric $d_0$ equal to the angular metric.

Let $C_\alpha'$ be obtained by deleting an edge $uv$ from $C_\alpha$. There
will be two components that are one-way infinite paths $P_1$ and $P_2$ starting
at $u$ and $v$; the rest is unchanged. The graphing $C_\alpha'$ is not compact
with the same metric, since the $1$-ball about $u$ (an edge) is not isomorphic
to the $1$-ball about any nearby point (two edges). The metric $d$, as
constructed above, remains the angular distance between any two points not on
$P_1\cup P_2$; the distance of points on each $P_i$ from points outside
$P_1\cup P_2$ will be positive (but decreasing as we go along the path). We can
think of lifting these countably many points off the cycle.

The completion of $C'_\alpha$ fills in the positions of the original points,
and the edges between them. So we obtain $C_\alpha$, together with two one-way
infinite paths spiralling closer and closer to it. The support of the measure
$\wh\lambda$ is just the original circle, and the restriction to this is just
the compact graphing $C_\alpha$.
\end{example}

\section{Properties of compact graphings}

On a compact graphing $\Gb=(I,d,E,\lambda)$, we have two metrics: the graph
distance $d_\Gb$ and the compact metric $d$. The following lemma shows that
they are behaving oppositely.

\begin{lemma}\label{LEM:CONTIN}
Let $\Gb=(I,d,E,\lambda)$ be a compact graphing. Then for every positive
integer $t$ there is an $\eps>0$ such that if $x,y\in I$ are at graph distance
$d_\Gb(x,y)\le t$, then $d(x,y)\ge \eps$.
\end{lemma}

\begin{proof}
Suppose not, then there is a sequence of pairs $(x_n,y_n)$ of distinct points
such that $d_\Gb(x_n,y_n)\le t$ and $d(x_n,y_n)\to 0$. We may assume (by the
compactness of $(I,d)$) that there is a point $x\in I$ such that
$d(x_n,x)\to0$. Clearly $d(y_n,x)\to 0$.

By the definition of compact graphings, if $n$ is large enough, then there is a
$t$-neighborhood isomorphism $\phi_n:~B(x_n,t)\to B(x,t)$ such that
$d(\phi_n)\to 0$. Let $y_n'=\phi_n(y_n)\in B(x,t)$. Since $B(x,t)$ is finite,
we may select a subsequence of the indices $n$ such that $y'_n=y'$ is
independent of $n$. Since $\phi_n$ is bijective, we must have $y'\not= x$. On
the other hand, $d(y_n,y')\le d(\phi_n)\to 0$, contradicting the fact that
$d(y_n,x)\to 0$.
\end{proof}

\begin{corollary}\label{COR:CONTIN}
For every compact graphing $\Gb=(I,d,E,\lambda)$ there is finite number of of
continuous measure preserving involutions $\sigma_1,\dots,\sigma_m$ such that
two distinct points $x,y\in I$ are adjacent if and only if $\sigma_i(x)=y$ for
some $i$.
\end{corollary}

\begin{proof}
By Lemma \ref{LEM:CONTIN} there is an $\eps>0$ such that the distance between
any two points at graph distance at most $2$ is at least $\eps$. By the
definition of compact graphings, there is a $\delta>0$ such that if
$d(x,y)<\delta$, then there is a $1$-neighborhood isomorphism $\phi_y$ from $x$
to $y$ such that $d(\phi_y)<\eps/4$. Fix a point $z\in N_\Gb(x)$. By the
definition of compact graphings, we have $\phi_y(z)\in\Bf(z,\eps/2)$, and so
$\phi_y(z)$ is a uniquely determined neighbor of $y$. This implies that the set
of edges $T_{x,z}=\{(y,\phi_y(z)):~y\in \Bf(x,\eps/4)$ is a matchingm and it is
is Borel. By Lemma 18.19 in \cite{HomBook}, the Borel graph $(I,T_{x,z})$ is a
graphing, and so the map interchanging the endpoints of edges in $T_{x_y}$ is
measure preserving for every fixed $x$ and $z$. It also follows by a similar
argument that this mapping is continuous.

By compactness, there is a finite number of open neighborhoods
$\Bf(x_1,\eps/4),\dots, \Bf(x_N,\eps/4)$ covering $I$. It follows that the edge
sets $T_{x_i,z}$ ($i=1,\dots,N$, $z\in N(x_i)$) cover all the edges.
\end{proof}

Lemma \ref{LEM:CONTIN} asserts that we have to walk through many edges of a
compact graphing to get back close to our starting node in the underlying
metric. The following generalization of the Poincar\'e Recurrence Theorem will
imply that walking sufficiently far, we {\it can} get back very close.

\begin{prop}\label{PROP:MASS-TR}
Let $\Gb=(I,\AA,E,\lambda)$ be a graphing, and let $A\in\AA$. Then
$|V(\Gb_x)\cap A|=\infty$ for almost every $x\in A$ for which
$|V(\Gb_x)|=\infty$.
\end{prop}

\begin{proof}
Fix an integer $r\ge 1$, and let $\Gb'$ denote the graph obtained from $\Gb$ by
connecting any two nodes at distance at most $r$. Then $\Gb'$ is a graphing (on
the same sigma-algebra of Borel sets; see e.g. \cite{HomBook}, Lemma 18.19).
The basic equation \eqref{EQ:UNIMOD} gives the identity
\begin{equation}\label{EQ:MASS2}
\int\limits_{U}|W\cap B(x,r)|\,d\lambda(x) = \int\limits_{W}|U\cap B(y,r)|\,d\lambda(y)
\end{equation}
for any two Borel sets $U,W\in \AA$.

Let $I'$ denote the union of infinite components of $\Gb$, and define
$A_k=\{x\in A\cap I':~|V(\Gb_x)\cap A|=k\}$ $(k=0,1,\dots)$. It is easy to see
that $I'$ and $A_k$ are Borel sets. It suffices to prove that $\lambda(A_k)=0$
for every $k$. Equation \eqref{EQ:MASS2} implies that
\[
\int\limits_{A_k}|B(x,r)|\,d\lambda(x) = \int\limits_{I}|A_k\cap B(y,r)|\,d\lambda(y)
\]
for every $k\ge 0$. The definition of $A_k$ implies that $|A_k\cap B(y,r)|\le
k$ for every $y\in I$. On the other hand, if $\Gb_x$ is infinite, we have
$|B(y,r)|\ge r+1$ for every $y\in V(\Gb_x)$. So the left side is at least
$(r+1)\lambda(A_k)$, while the right side is at most $k$. Letting $r\to\infty$,
we get that $\lambda(A_k)=0$.
\end{proof}

Let us call a component $H$ of a compact graphing $\Gb$ {\it self-dense}, if
every point in $x\in V(H)$ is a limit point of $V(H)\setminus\{x\}$, and {\it
self-avoiding} if every point $x\in V(H)$ is at a positive distance from
$V(H)\setminus\{x\}$.

\begin{lemma}\label{LEM:SELF-DENSE}
Every connected component in a compact graphing is either self-dense or
self-avoiding.
\end{lemma}

\begin{proof}
It suffices to prove that if for $x\in I$ there is an infinite sequence of
distinct points $x_1,x_2,\dots\in V(H)$ such that $x_n\to x$, then this holds
for every neighbor $y$ of $x$. By the definition of compact graphings, once $n$
is large enough, there is a $1$-neighborhood isomorphism $\phi_n)$ between $x$
and $x_n$ such that $d(\phi_n)\to 0$. So $y_n=\phi_n(y)$ is a neighbor of $x_n$
such that $d(y,y_n)\to 0$ as $n\to\infty$. Trivially $y_n\in V(H)$, and at most
$D$ of the $y_n$ can be equal, since every $y_n$ can have at most $D$ neighbors
among the $x_i$. So can select an infinite subsequence of the points $y_n$
tending to $y$.
\end{proof}

Clearly every finite component is self-avoiding. I don't know whether a compact
graphing can have self-avoiding infinite components at all. But at least we can
prove that almost all infinite components are self-dense:

\begin{prop}\label{PROP:GET-BACK}
Let $\Gb=(I,\AA,E,\lambda)$ be a compact graphing. Then for almost all $x\in
I$, the connected component $\Gb_x$ is either finite or self-dense.
\end{prop}

\begin{proof}
Let, as above, $I'$ denote the union of infinite components of $\Gb$, and let
$C=\{x\in I':~\Gb_x\text{ is self-avoiding}\}$. Then $\delta(x)=d(x,
V(\Gb_x)\setminus\{x\})>0$ for every $x\in C$. Fix an $\eps>0$. It is easy to
see that the set $U_\eps=\{x\in C:~\delta(x)\ge\eps\}$ is Borel. For every
$x\in U$, the points in $V(\Gb_x)\cap U_\eps$ are at least $\eps$ apart, and
hence by the compactness of $(I,d)$, the intersection $V(\Gb_x)\cap U_\eps$ is
finite. By Proposition \ref{PROP:MASS-TR}, this implies that
$\lambda(U_\eps)=0$. Since this holds for all $\eps>0$, it follows that
$\delta(x)=0$ almost everywhere, i.e., $\lambda(C)=0$.
\end{proof}

\section{Concluding remarks}

Our construction of the metric \eqref{EQ:DIST2} is similar to the metric on
random rooted graphs introduced by Benjamini and Schramm \cite{BS2}: that would
essentially correspond to $d_0\equiv0$. Example \ref{EXA:CIRC} illustrates the
difference: the Benjamini--Schramm distance of any two points in $C_\alpha$ is
zero, while we construct a proper metric.

In the theory of dense graph limits, the purification/compactification
procedures result in a ``canonical'' compact graphon locally equivalent to the
graphon we start with. In the bounded-degree case, the compactification of a
graphing described in this note depends on the metric on its underlying set we
start with. Can we make the compactified graphing ``canonical'', i.e., not
dependent on arbitrary choices?

An important use of the purification of a graphon is the reasonable definition
of its automorphism group and the proof of its compactness \cite{LSz15}. Is
there a way to compactify a graphing so that its automorphism group (defined in
a reasonable way) is preserved? Will this automorphism group be compact in a
suitable topology?

\end{document}